\documentclass[12pt,reqno,dvipsnames]{amsart}
\usepackage[a4paper, margin=1.25in]{geometry}

\usepackage{amssymb,amsmath,amsthm,amscd,epsfig}
\usepackage[linktocpage=true,backref]{hyperref}
\usepackage{mathrsfs}
\usepackage{mathtools}
\usepackage{tikz-cd}
\usepackage{enumitem}
\usepackage{soul}

\usepackage{bm}

\newtheorem{thm}{Theorem}[section]

\newtheorem{lemma}[thm]{Lemma}
\newtheorem{theorem}[thm]{Theorem}
\newtheorem{proposition}[thm]{Proposition}
\newtheorem{corollary}[thm]{Corollary}

\theoremstyle{definition}

\newtheorem{note}[thm]{Note}

\newtheorem{definition}[thm]{Definition}

\newtheorem{remark}[thm]{Remark}
\newtheorem{mydef}[thm]{Definition}
\newtheorem*{UnNot}{(Unusual) Choice of notation}

\numberwithin{equation}{section}

\theoremstyle{remark}

\numberwithin{table}{section}

\def\Z{\ifmmode{{\mathbb Z}}\else{${\mathbb Z}$}\fi}
\def\F{\ifmmode{{\mathbb F}}\else{${\mathbb F}$}\fi}
\def\Q{\ifmmode{{\mathbb Q}}\else{${\mathbb Q}$}\fi}
\def\C{\ifmmode{{\mathbb C}}\else{${\mathbb C}$}\fi}
\def\P{\ifmmode{{\mathbb P}}\else{${\mathbb P}$}\fi}
\def\H{\ifmmode{{\mathrm H}}\else{${\mathrm H}$}\fi}
\def\G{\ifmmode{{\mathbb G}}\else{${\mathbb G}$}\fi}
\def\R{\ifmmode{{\mathbb R}}\else{${\mathbb R}$}\fi}
\def\F{\ifmmode{{\mathbb F}}\else{${\mathbb F}$}\fi}
\def\N{\ifmmode{{\mathbb N}}\else{${\mathbb N}$}\fi}
\def\O{\ifmmode{{\calO}}\else{${\calO}$}\fi}
\def\D{\ifmmode{{\cal{D}}^b}\else{${{\cal{D}}^b}$}\fi}

\newcommand{\calO}{{\mathcal O}}

\newcommand{\restricts}[2]{#1|_{#2}}

\DeclareMathOperator{\Char}{char}

\DeclareMathOperator{\Bad}{Bad}

\DeclareMathOperator{\Gal}{Gal}
\DeclareMathOperator{\Val}{Val}

\DeclareMathOperator{\Pic}{Pic}

\DeclareMathOperator{\Spec}{Spec}
\DeclareMathOperator{\DVR}{DVR}

\DeclareMathOperator{\Frac}{Frac}

\DeclareMathOperator{\id}{id}

\usepackage{todonotes}

\DeclareMathOperator{\Br}{Br}

\providecommand{\Pic}{\mathrm{Pic}}

\let\originalleft\left
\let\originalright\right
\renewcommand{\left}{\mathopen{}\mathclose\bgroup\originalleft}
\renewcommand{\right}{\aftergroup\egroup\originalright}

\makeatletter
\@namedef{subjclassname@2020}{%
  \textup{2020} Mathematics Subject Classification}
\makeatother

\title{Weak approximation for del Pezzo surfaces of low degree}
\subjclass[2020]{14G05, 14G12 (11G35).}

\author{Julian Demeio}
\address{Julian Demeio \\
Max Planck Institute for Mathematics \\
Vivatsgasse 7 \\
53111 Bonn \\
Germany.}
\email{demeio@mpim-bonn.mpg.de}

\author{Sam Streeter}
\address{Sam Streeter \\
School of Mathematics \\
University of Bristol \\
Woodland Road \\
Bristol \\
BS1 8UG \\
UK.}
\email{sam.streeter@bristol.ac.uk}

\begin{document}
\begin{abstract}
We prove, via an ``arithmetic surjectivity'' approach inspired by work of Denef, that weak weak approximation holds for surfaces with two conic fibrations satisfying a general assumption. In particular, weak weak approximation holds for general del Pezzo surfaces of degrees $1$ or $2$ with a conic fibration.
\end{abstract}
\maketitle
\setcounter{tocdepth}{1}
\tableofcontents
\section{Introduction}
Questions regarding rational points on del Pezzo surfaces (smooth projective surfaces $X$ with ample anticanonical divisor $-K_X$) are of long-standing interest in Diophantine geometry. Two guiding principles are that these surfaces should have many rational points (if they have any) and that their complexity should increase as the degree $d := K_X^2 \in \{1,\dots,9\}$ decreases. As such, the cases $d=1,2$ are thought of as the most challenging, and they will be the focus of this paper.

One measure of the geometric abundance and distribution of rational points for varieties over a number field is \emph{weak approximation} (Definition \ref{def:WA}). Kresch and Tschinkel \cite{KT} and V\'arilly-Alvarado \cite{VAR} have provided counterexamples to weak approximation for del Pezzo surfaces of degrees $2$ and $1$ respectively. However, it is conjectured by Colliot-Th\'el\`ene that any smooth unirational variety satisfies \emph{weak weak approximation}, or weak approximation away from finitely many places, and all del Pezzo surfaces with a rational point are expected to be unirational.

We shall prove that del Pezzo surfaces with a rational point over a number field satisfy weak weak approximation, provided we assume the additional structure of two conic fibrations. While this requirement may seem demanding, any conic fibration on a del Pezzo surface of degree $1,2$ or $4$ gives rise to another (Lemma \ref{isklem}). Further, Iskovskih \cite{ISK} showed that any geometrically rational surface is birational to either a del Pezzo surface or a conic bundle, so one may think of del Pezzo surfaces with conic fibrations as a special ``intersectional'' case towards the proof of weak weak approximation for all geometrically rational surfaces. Indeed, degree-$d$ del Pezzo surfaces with a rational point are rational (and so satisfy the stronger property of weak approximation) if $d \geq 5$, while weak weak approximation follows from work of Salberger and Skorobogatov \cite{SS} and Swinnerton-Dyer \cite{SD} for the cases $d=4$ and $d=3$ respectively. For surfaces $X$ with a conic fibration, the arithmetic complexity is notionally controlled by the number of singular fibers. Indeed, if $\pi: X \rightarrow \mathbb{P}^1$ is a conic fibration with $r$ singular fibers, then one obtains (by blowing down components of singular fibers over the algebraic closure) the formula $K_X^2 = 8 - r$, hence $X$ is a del Pezzo surface of degree $8-r$ when $r \leq 7$. Work of Koll\'ar and Mella \cite{KM} shows that, for $r \leq 7$, the surface $X$ is unirational as soon as it possesses a rational point. The arithmetic of conic bundles (surfaces with conic fibrations) is connected to \emph{Schinzel's hypothesis}, the verification of which is one of the foremost open problems in number theory. In particular, assuming the truth of Schinzel's hypothesis implies that the Brauer--Manin obstruction is the only one to weak approximation for surfaces with a conic fibration \cite[Thm.~14.2.5]{BGbook}. 
The finiteness of the quotient $\Br X/\Br K$ (which follows from \cite[Prop.\ 11.3.4]{BGbook}) then implies that weak weak approximation holds as soon as the surface possesses a rational point by \cite[Lem.\ 13.3.13]{BGbook}.
For details on the connection between Schinzel's hypothesis and rational points, we refer the reader to \cite[\S14.2]{BGbook}. 
\subsection{Results}
\begin{theorem} \label{mainthm}
Let $X$ be a smooth projective surface over a number field $K$ with two distinct conic fibrations $\pi_i: X \rightarrow \mathbb{P}^1$, $i=1,2$. Suppose that $X\left(K\right) \neq \emptyset$ and that the following conditions hold.
\begin{enumerate}[label=(\alph*)]
\item Fibers from distinct fibrations have non-zero intersection product.
\item Singular fibers from distinct fibrations do not share a singular point.
\end{enumerate}
Then $X$ satisfies weak weak approximation over $K$.
\end{theorem}

\begin{note}
Our result does not apply only to del Pezzo surfaces. However, when $X$ is \emph{minimal} and satisfies the hypotheses of Theorem \ref{mainthm}, it is necessarily a del Pezzo surface. Indeed, any minimal rational surface with two conic fibrations satisfies $K_X^2>0$ by \cite[Thm.~1.6]{ISK70} and then it is del Pezzo by \cite[Thm.~5]{ISK} (the $N$th scroll $F_N$ possesses, even geometrically, only one conic fibration for $N\geq 2$).
\end{note}

\begin{remark}
Since a conic fibration on a del Pezzo surface of degree $d \in \{1,2,4\}$ induces another (Lemma \ref{isklem}), we may think of Theorem \ref{mainthm} as holding for general del Pezzo surfaces of degrees $1$ or $2$ with a conic fibration (see Remark \ref{GBsp}).
\end{remark}

\begin{remark}
By \cite[Prop.~V.1.4]{HAR} and \cite[Rem.~II.7.8.1]{HAR}, the first condition is equivalent to the two conic fibrations not being equal up to an automorphism of $\mathbb{P}^1$, hence one may think of it as ensuring that the two fibrations are truly distinct.
\end{remark}

We deduce the following result for del Pezzo surfaces of degrees $1$ or $2$.
\begin{corollary}
Let $X$ be a del Pezzo surface of degree $1$ or $2$ with a conic fibration and a rational point. If there exists no $4$-Eckardt point in $X({\overline{K}})$, then $X$ satisfies weak weak approximation.
\end{corollary}
For the definition of $4$-Eckardt points, see Definition \ref{gepdef}.

By judiciously blowing up, our method gives a new proof of the following result.
\begin{corollary} \label{cor:dp4}
Let $X$ be a del Pezzo surface of degree $4$ containing a rational point. Then $X$ satisfies weak weak approximation.
\end{corollary}
The proof of Theorem \ref{mainthm} builds upon that of \cite[Thm.~1.2]{STR}, which uses the two conic fibrations $\pi_1$ and $\pi_2$ to generate rational points. Starting with a point $P \in X(K)$ on a smooth fiber of $\pi_1$ (the existence of which follows, possibly after relabelling, from condition (b)), the curve $\pi_1^{-1}(\pi_1(P))$ is a smooth conic with a rational point, hence it is isomorphic to $\P^1_K$ and therefore has many rational points. By repeating this process with the smooth fibers of $\pi_2$ through the rational points of this curve, we produce yet more rational points. The second author proves in \cite{STR} that the $K$-points generated through this method are enough to prove the \emph{Hilbert property} on $X$, i.e.\ that $X\left(K\right)$ is not thin. In this paper we prove (under the hypotheses of Theorem \ref{mainthm}) that by taking this iterative process to its fifth step, we may verify weak weak approximation on $X$, which is stronger than the Hilbert property \cite[Thm.~3.5.7]{SER}.

The proof goes by constructing, for each $n \geq 1$, auxiliary \emph{rational} varieties $C'_n$ over $K$ endowed with morphisms $f'_n:C'_n \to X$, and such that the image of the $K$-points of $C'_n$ contain the $K$-points obtained through the iterative process described above. Applying a result of Denef \cite{DEN}, we show that $f'_n$ is \emph{arithmetically surjective} for all $n \geq 5$, meaning that for sufficiently large places $v$ of $K$, $f'_n(C'_n(K_v))=X(K_v)$. Using the fact that $C'_n$ is a smooth proper \emph{rational} variety, and so satisfies weak approximation, we conclude that weak weak approximation holds on $X$.

To apply the aforementioned result of Denef, the main difficulty lies in proving that every birational modification $\widetilde{C'_5} \to \widetilde{X}$ of $C'_5 \to X$ has split fibers over the codimension-one points of $\widetilde{X}$, and the establishment of this fact (Proposition \ref{MainProp}) is the technical heart of our proof.
\subsection{Conventions}
All fields appearing in the paper are assumed to have characteristic $0$.
A \emph{variety} over a field $k$ is a geometrically integral separated scheme of finite type over $k$. A curve (resp.\ surface) is a variety of dimension $1$ (resp.\ $2$).
\begin{mydef}
Let $X$ be a smooth projective surface over a field $k$. A \emph{conic fibration} of $X$ is a morphism $\pi: X \rightarrow \P^1$ such that all fibers of $\pi$ are isomorphic to plane conics.
\end{mydef}

Note that all fibers of a conic fibration are automatically reduced: indeed, the unique irreducible component of a non-reduced fiber would have, by the adjunction formula, non-integral genus $\frac{1}{2}$. 

\begin{mydef}
Let $X$ be a smooth projective surface over a field $k$, and let $C \subset X_{\overline{k}}$ be a geometrically rational curve. We say that $C$ is an \emph{exceptional curve} if $C^2 = -1$.
\end{mydef}
Note that, by the adjunction formula, a geometrically rational curve on a del Pezzo surface is an exceptional curve if and only if it has self intersection $-1$.
\subsection{Notation}
Given a number field $K$, we denote by $\Val\left(K\right)$ the set of places of $K$. Given a morphism of schemes $f:X \rightarrow Y$ and a morphism $Z \rightarrow Y$, we denote by $f_Z:X_Z \rightarrow Z$ the base change $X \times_Y Z \rightarrow Z$.
\section{Background}
\subsection{Weak approximation}
\begin{mydef} \label{def:WA}
Let $X$ be a variety over a number field $K$. We say that $X$ satisfies \emph{weak weak approximation} if there exists a finite subset of places $S \subset \Val\left(K\right)$ such that $X\left(K\right)$ is dense in $\prod_{v \not\in S}X\left(K_v\right)$. We say that $X$ satisfies \emph{weak approximation} if we may take $S = \emptyset$.
\end{mydef}
Weak approximation can be thought of as both an indicator that rational points are well-distributed (among the local points) and that they are numerous (there are enough of them so as to be dense in the local points). Weak approximation is satisfied by $\mathbb{A}^n_K$ and is a birational invariant of smooth varieties, hence all smooth rational varieties satisfy weak approximation. Further, it is connected to the notion of \emph{thin sets}.
\begin{mydef}
Let $X$ be a variety over a field $k$ and let $A \subset X\left(k\right)$. We say that $A$ is \emph{thin} if there exists a finite collection of dominant finite morphisms $f_i: Y_i \rightarrow X$ of degree $>1$, $i=1,\dots,r$ from geometrically irreducible normal varieties $Y_i$ such that $A \setminus \cup_{i=1}^r f_i\left(Y_i\left(k\right)\right)$ is not Zariski-dense in $X$.
\end{mydef}
Work of Colliot-Th\'el\`ene and Ekedahl \cite[Thm.~3.5.7]{SER} shows that weak weak approximation implies that the variety in question possesses the \emph{Hilbert property}, meaning that the set $X\left(K\right)$ itself is not thin. The main result of \cite{STR} is that the varieties in Theorem \ref{mainthm} satisfy the Hilbert property not only over number fields, but over any Hilbertian field (that is to say, any field over which there exists a variety with the Hilbert property). It is not known whether every variety with the Hilbert property satisfies weak weak approximation, although it is suggested by Corvaja and Zannier \cite[\S1.5]{CZ} that this may not be the case.
\subsection{Geometrically rational surfaces}
In \cite{ISK}, Iskovskih showed that, over an arbitrary ground field $k$, any smooth geometrically rational surface is $k$-birational to either a conic bundle (a surface with a conic fibration) or a del Pezzo surface. Thus, when exploring properties such as weak weak approximation which are invariant under birational transformations of smooth varieties, these surfaces are of particular significance.

In this paper, we focus on surfaces lying within the intersection of these two families, i.e.\ del Pezzo surfaces with a conic fibration. It follows from work of Koll\'ar and Mella \cite[Cor.~8]{KM} that these surfaces are unirational as soon as they possess a rational point. For our methods, we will require two conic bundle structures, but the following result of Iskovskih shows that this is not too much to ask.

\begin{lemma}{\cite[Proof~of~Thm.~5]{ISK}} \label{isklem}
Let $X$ be a del Pezzo surface over a field $k$ of degree $d \in \{1,2,4\}$, and let $\pi: X \rightarrow \mathbb{P}^1$ be a conic fibration. Let $C \in \Pic\left(X\right)$ be the linear equivalence class of the fibers of $\pi$. Then the class $C' = -\frac{4}{d}K_X - C$ gives rise to another conic fibration $\pi': X \rightarrow \mathbb{P}^1$, and $C \cdot C' = \frac{8}{d}$.
\end{lemma}

\begin{remark}
From the above, we see that any del Pezzo surface of degree $d \in \{1,2,4\}$ containing a (possibly singular) curve $C$ with $C^2 = 0$ and $-K_X \cdot C = 2$ is endowed with two conic fibrations, and that fibers from the two fibrations have non-zero intersection product (so condition (a) of Theorem \ref{mainthm} is satisfied). Further, on \emph{minimal} del Pezzo surfaces, i.e.\ those containing no exceptional curves over the ground field, it follows from \cite[Thm.~1]{ISK} that there exist at most two conic fibrations. As such, while the hypotheses of Theorem \ref{mainthm} ostensibly offer greater flexibility in the choice of the two conic bundles, the scenario in which the two conic bundles are ``dual'' as in Lemma \ref{isklem} should be thought of as the main case.
\end{remark}

It follows from the adjunction formula that the singular fibers of a conic fibration of a del Pezzo surface consist of two exceptional curves meeting in a point, and so condition (b) of Theorem \ref{mainthm} concerns points which lie on four exceptional curves. Such points play an important role in the arithmetic of del Pezzo surfaces of low degree, as seen in \cite{STVA}, hence we make the following definition.

\begin{mydef} \label{gepdef}
Let $X$ be a del Pezzo surface over $K$ and let $n$ be a positive integer. We say that $P \in X(\overline{K})$ is an \emph{$n$-Eckardt point} if it lies on at least $n$ exceptional curves of $X_{\overline{K}}$.
\end{mydef}

The above terminology arises from the study of cubic surfaces (del Pezzo surfaces of degree $3$), where points on three exceptional curves are known as Eckardt points. In the degree-$2$ case, points on four exceptional curves are known as generalized Eckardt points (see \cite[\S2.2]{STVA}). In each case, such points are distinguished by the fact that they lie on the maximum possible number of exceptional curves. However, on a del Pezzo surface of degree $1$ over a field of characteristic zero, the maximum number of concurrent exceptional curves is ten (see \cite[Thm.~1.2]{vLW}). While $n$-Eckardt points do not necessarily satisfy the ``maximality'' property of their predecessors, they at least share with them the property of lying on many exceptional curves.

\section{Auxiliary varieties}\label{Sec3}
Let $X$ be a smooth projective surface over a field $k$. Let $\pi_1,\pi_2: X \rightarrow \P^1$ be two conic fibrations such that $\pi_1^{-1}\left(P\right) \cdot \pi_2^{-1}\left(Q\right) > 0$ for all $P,Q \in \mathbb{P}^1(\overline{k})$. Suppose that there exists $P_0 \in X\left(k\right)$ such that the fiber $\pi_1^{-1}\left(\pi_1\left(P_0\right)\right)$ is smooth (as noted earlier, this holds, possibly after relabelling, as soon as $X\left(k\right) \neq \emptyset$ for $X$ satisfying the hypotheses of Theorem \ref{mainthm}).
 
We begin by introducing fiber products used in \cite{STR} to propagate rational points on $X$, which will be central to our proof.
\begin{mydef} \label{DefCn}
Set $C_0 = \{P_0\}$, and denote by $f_0: C_0 \rightarrow X$ the inclusion of $C_0$ in $X$. For $n \geq 1$, define $C_n$ to be the fiber product
\begin{center}
\begin{equation}\label{DiagramCn}
\begin{tikzcd}
C_n \arrow["a_n",d] \arrow["f_n",r] & X \arrow["\pi_i",d] \\
C_{n-1} \arrow{r}[swap]{\pi_i \circ f_{n-1}} & \mathbb{P}^1 ,
\end{tikzcd}
\end{equation}
\end{center}
where $i = 1$ if $n$ is odd and $i=2$ if $n$ is even.
\end{mydef}
From the rational point $C_0$, we produce the fiber $C_1 = \pi_1^{-1}\left(\pi_1\left(P_0\right)\right)$ with infinitely many rational points (since it is a conic with a smooth rational point), and in the next iteration, we produce $C_2$, whose rational points correspond to pairs $\left(P_1,P_2\right)$ with $P_1 \in C_1\left(k\right)$, $P_2 \in \pi_2^{-1}\left(\pi_2\left(P_1\right)\right)\left(k\right)$. In particular, for each of the infinitely rational points $P \in C_1\left(k\right)$ such that $\pi_2^{-1}\left(\pi_2\left(P_1\right)\right)$ is smooth, we have infinitely many rational points on $X$ lying on this fiber (there are infinitely many such fibers, since $\pi_2|_{C_1}:C_1 \to \P^1$ is dominant). We may think of these fiber products as giving rise to many rational points on $X$. Indeed, the rational points coming from the second iteration are Zariski-dense in $X$. We will show that, upon further iterating this process, the task of proving weak weak approximation on $X$ can be translated to $C_n$.

\begin{remark}\label{RmkExplicitCn}
The following explicit description of $C_n, n \geq 1$ will be useful in Section \ref{Sec:GBsp}, in which we show that condition (b) of Theorem \ref{mainthm} is necessary for our method:
\begin{equation}\label{Eq:notationforCn}
    C_n = 
\begin{cases}
\overbrace{X \times_{\pi_1,\pi_1} X \times_{\pi_2,\pi_2} \cdots \times_{\pi_2,\pi_2}X }^{n} \times_{\pi_1,f_0\circ \pi_1} C_0 \ & \text{ if } n \text{ is odd }, \\
\overbrace{X \times_{\pi_2,\pi_2} X \times_{\pi_1,\pi_1} \cdots \times_{\pi_2,\pi_2} X}^n \times_{\pi_1,f_0\circ \pi_1} C_0 \ & \text{ if } n \text{ is even }.
\end{cases}
\end{equation}

Here, the two morphisms under each sign ``$\times$'' indicate the two morphisms with respect to which we take the fiber product over $\mathbb{P}^1$. Under the above identification, the morphism $f_n$ corresponds to the projection on the first factor.
\end{remark}

\begin{proposition} \label{PropGeomint}
Let $C_n$ and $X$ be defined as above.
\begin{enumerate}[label=(\roman*)]
\item For $n \geq 2$, the morphism $f_n: C_n \rightarrow X$ is flat, projective and surjective.
\item For $n \geq 0$, the scheme $C_n$ is geometrically integral of dimension $n$.
\item For $n \geq 0$, $C_n$ is a rational projective variety.
\end{enumerate}
\end{proposition}
\begin{proof}
\begin{enumerate}[label=(\roman*)]
\item Since flatness, projectivity and surjectivity are preserved under base change and composition, it suffices to verify that $\pi_2 \circ f_1:C_1 \rightarrow \mathbb{P}^1$ is flat, projective and surjective. By miracle flatness \cite[Exercise~III.10.9]{HAR}, the morphism $\pi_2 \circ f_1$ is either flat or constant. In the latter case, we would have that $C_1$ is contained in a fiber of $\pi_2$, which would imply that $C_1\cdot \pi_2^{-1}(P)=0$ for every $P \in \P^1$. By hypothesis $(a)$ of Theorem \ref{mainthm}, this cannot hold, hence we deduce flatness. As the composition of projective morphisms, $\pi_2 \circ f_1$ is also projective, and as a non-constant morphism of smooth projective curves, it is surjective. 

\item Since the formation of the $C_n$ commutes with base change of the field, we may assume that $k$ is algebraically closed, hence we need only prove that $C_n$ is integral. We prove this by induction on $n \geq 1$ (the claim being trivial for $n=0$). 

$(n=1)$: Since $C_1$ is smooth and connected, it is integral. 

$(n-1\Rightarrow n, n \geq 2)$: Assume that $C_{n-1}$ is integral.
Note that each $\pi_i$ is flat with geometrically integral generic fiber. These two properties are preserved under surjective base change, so $a_n: C_n \to C_{n-1}$ satisfies them as well. Applying \cite[Prop.~4.3.8]{LIU} to $a_n$, we deduce that $C_n$ is integral.

\item Projectivity follows from projectivity of $f_n$, so it remains to prove rationality. Again, we induct on $n$. Note that the cases $n=0,1$ are trivial.

Assume that $C_{n-1}$ is rational. Note that the morphism $\pi_i \circ f_{n-1}: C_{n-1} \rightarrow \P^1$ is surjective for every $n\geq 2$, hence it sends the generic point of $C_{n-1}$ to the generic point of $\P^1$. It follows that the geometric generic fiber of $C_n \to C_{n-1}$, being a base change of the geometric generic fiber of $\pi_i:X \to \P^1$, is isomorphic to $\P^1$. Note, moreover, that the morphism $C_n \to C_{n-1}$ has a natural section induced by $\id_{C_{n-1}}$ and $f_{n-1}$. Therefore the generic fiber of $C_n \to C_{n-1}$ is a form of $\P^1$ with a rational point, hence it is isomorphic to $\P^1$. Since $C_{n-1}$ is rational and $C_n$ is integral, this implies that $C_n$ is rational as well. \qedhere
\end{enumerate}
\end{proof}

Now let $k=K$, a number field.

Let $C'_n \rightarrow C_n$ be a desingularization of $C_n$, and $f'_n$ be the composition $C'_n \rightarrow C_n \xrightarrow{f_n} X$.
Note that $C'_n$ is smooth and rational, hence it satisfies weak approximation. Therefore, in order to verify that $X$ satisfies weak weak approximation, it suffices to show that for all but finitely many places $v \in \Val\left(K\right)$, the map $f'_n:C'_n\left(K_v\right) \rightarrow X\left(K_v\right)$ is surjective.

\begin{lemma}\label{Lem:gigf}
For $n \geq 3$, the generic fiber of $f_n$ is geometrically integral, hence so is that of the composition $C'_n \to C_n \xrightarrow{f_n} X$.
\end{lemma}

\begin{proof}
First, observe that $f_n$ is flat for all $n \geq 2$. Indeed, $\pi_2\circ f_1: C_1 \rightarrow \mathbb{P}^1$ is a finite morphism of smooth curves, hence flat by miracle flatness. Since flatness is preserved under base change, flatness of $f_2$ follows from that of $\pi_2\circ f_1$. Moreover, flatness of $f_{n+1}, n \geq 2$ follows from flatness of $f_n$ since flatness is preserved under base change and composition. For flat proper morphisms, having geometrically integral fibers is an open condition \cite[Thm.~12.2.4(viii)]{EGA4}, hence it suffices to show that, for $n \geq 3$, there exists $P \in X$ such that $f_n^{-1}\left(P\right)$ is geometrically integral.

Let us first consider the case $n=3$. Given $P \in X$, one may identify $f_3^{-1}\left(P\right)$ with the fiber of $\pi_1 \circ f_2: C_2 \rightarrow \mathbb{P}^1$ over $\pi_1\left(P\right)$, and in turn with the fiber product of $\pi_1^{-1}\left(\pi_1\left(P\right)\right)$ and $\pi_1^{-1}\left(\pi_1\left(P_0\right)\right)$ mapping to $\mathbb{P}^1$ under $\pi_2$. A sufficient condition for (geometric) integrality of this fiber product is that $\pi_1^{-1}\left(\pi_1\left(P\right)\right)$ is geometrically irreducible (which is true for $P$ not lying on any of the singular fibers of $\pi_1$) and the two morphisms to $\mathbb{P}^1$ have disjoint branch loci (see \cite[Lem.~2.8]{STR}). On the other hand, $Q \in \mathbb{P}^1$ is a common branch point if and only if $\pi_2^{-1}\left(Q\right)$ intersects both $\pi_1^{-1}\left(\pi_1\left(P\right)\right)$ and $\pi_1^{-1}\left(\pi_1\left(P_0\right)\right)$ non-transversally. Let $r$ be the intersection product of fibers from distinct fibrations. Since all fibers of the $\pi_i$ are reduced, it follows from the Riemann--Hurwitz formula \cite[Cor.~IV.2.4]{HAR} that each fiber of $\pi_2$ intersects at most $2r-2$ fibers of $\pi_1$ non-transversally. Then, for $P$ chosen outside some finite union of fibers of $\pi_1$, this fiber product is integral. We deduce the result for $n=3$.

Now we establish the induction step. Let $F_n$ be the generic fiber of $f_n$, and assume that it is geometrically integral. Let $E_n$ be the generic fiber of $\pi_i \circ f_n: C_n \rightarrow \mathbb{P}^1$. By combining Cartesian squares, we see that $F_{n+1} \cong E_n \times_{k\left(\mathbb{P}^1\right)} k\left(X\right)$, hence it suffices to verify that $E_n$ is geometrically integral. In turn, letting $D_i$ be the generic fiber of $\pi_i$, a smooth conic over $k\left(\mathbb{P}^1\right)$, we have $E_n \cong C_n \times_{X} D_i$, and the generic fiber of $E_n \rightarrow D_i$ is isomorphic to $F_n$. Since $F_n$ and $D_i$ are geometrically integral and the morphism $E_n \rightarrow D_i$ is flat as the base change of the flat morphism $f_n: C_n \rightarrow X$, it follows from \cite[Prop.~4.3.8]{LIU} that $E_n$ is geometrically integral, hence $F_{n+1}$ is also geometrically integral.
\end{proof}

\section{Splitness}

In this section we give a notion of ``split reduction'' for surjective proper morphisms between $k$-varieties $f:Y \rightarrow X$,
stemming from the following definition which first appeared in \cite[Def.~0.1]{SKO}.

\begin{definition}\label{Def:Splitvariety}
We say that a scheme $X$ of finite type over a perfect field $F$ is \emph{split} if there exists a geometrically integral open subscheme $U \subset X$.
\end{definition}
\subsection{Split schemes over DVRs}
For a field $K$ and a subfield $k \subset K$, we define the following set of discrete valuation rings (DVRs):
\[
\DVR(K,k):=\{\text{DVRs } R \mid k \subseteq R \subseteq K , \ \Frac(R)=K\}.
\]
If $K$ is a finitely generated field over $k$, we say that $R \in \DVR(K,k)$ is \emph{divisorial} if there exists a normal $k$-variety $X$ such that its fraction field $k(X)$ is isomorphic to $K$, and a codimension-$1$ point $\eta \in X$ such that $R$ is the image of the DVR $\mathcal{O}_{X,\eta} \subseteq k(X)$ under the isomorphism $k(X) \cong K$. We denote the set of divisorial DVRs $R \in \DVR\left(K,k\right)$ by 
$\DVR'(K,k)$.
Finally, for a $k$-variety $X$, we use the following notation:
\[
\DVR(X):=\left\{R \in \DVR'(k(X),k) \left|  
	\begin{gathered}
	\text{There exists a commutative diagram }  \\
	\hspace{-0.5cm}\begin{tikzcd}[column sep=tiny, row sep= tiny]
	&[-10pt] \Spec R \arrow[rd] &   \\
	\Spec k(X)\arrow[rr] \arrow[ru] &                    & X
	\end{tikzcd}  
	\end{gathered}\right.
	\right\}.
\]

Note that, when $X$ is proper, we have $\DVR\left(X\right) = \DVR'\left(k\left(X\right),k\right)$.

The following definition seems to have never appeared explicitly in the literature, although it is implicit in several works (see e.g. \cite{SKO}, \cite[\S 3]{CT08}, \cite{DEN} or \cite[Ch.\ 9]{BGbook}).

\begin{mydef} \label{def:splitred}
Let $K$ be a finitely generated field over a field $k$ of characteristic zero. Let $Y$ be a $K$-variety, and let $R \in \DVR\left(K,k\right)$. We say that $Y$ has \emph{split reduction at $R$} if for some regular integral proper $R$-scheme $\mathcal{Y}$ with generic fiber smooth and birational to $Y$, the special fiber of $\mathcal{Y}$ is split.

We say that a proper surjective morphism $f:Y \to X$ between $K$-varieties has split reduction at $R \in \DVR (k(X),k)$ if its generic fiber does.
\end{mydef}

\begin{note}
By standard desingularization results, an $R$-scheme $\mathcal{Y}$ as in Definition \ref{def:splitred} exists for $R$ divisorial. Further, for $X$ and $Y$ smooth, $f: Y \rightarrow X$ has split reduction at $R = \mathcal{O}_{X,\eta} \in \DVR(X)$ (where $\eta$ is a codimension-$1$ point of $X$) if and only if the fiber of $f$ over $\eta$ is split. By \cite[Lem.~1.1]{SKO}, the special fiber of $\mathcal{Y}$ is split if and only if there exists a residually closed local flat extension of DVRs $i: R \rightarrow R'$ of ramification index one such that the generic fiber of $\mathcal{Y}$ has a $\Frac\left(R'\right)$-point. By Nishimura's lemma \cite[Thm.~3.6.11]{POO}, we see that the above definition is independent of the choice of $R$-model $\mathcal{Y}$. 
\end{note}

\subsection{Split fibers}

\begin{proposition}\label{Prop:SplitImpliesSplit}
Let $k$ be a field with $\Char k=0$, and $f:Y \rightarrow X$ be a proper morphism of  $k$-varieties. Assume that there exists an open subscheme $U \subseteq Y$ such that $f|_U: U \to X$ is smooth and has split fibers. Then, for every $R \in \DVR(f(U))\subseteq \DVR(X)$, $f$ has split reduction at $R$.
\end{proposition}
\begin{proof}
Since $f|_U$ is smooth, it is flat, hence $f(U)$ is open. In particular, after restricting $X$ to $f(U)$, we may assume that $f$ is surjective.
Let $R \in \DVR(X)$ and $\xi \in \Spec R$ be the special point. We recall that $R \in \DVR(X)$ means that $R$ is a divisorial DVR containing $k$, whose fraction field is $k(X)$ and such that there exists a (necessarily unique) morphism 
$\phi:\Spec R \rightarrow X$ whose restriction to the generic point of $\Spec R$ is the natural morphism $\Spec k(X) \rightarrow X$. 

We have an open embedding $U \times_X \Spec R \subset Y\times_X \Spec R$. Note that $U \times_X \Spec R$, being smooth over the regular ring $R$, is regular, hence there exists a desingularization $Y' \rightarrow Y\times_X \Spec R$ that is an isomorphism over $U \times_X \Spec R$. In particular, the special fiber $Y'\times_R \xi$ contains the open subscheme $U' := U \times_X \xi$, which is non-empty since $\phi(\xi) \in f(U)$. Moreover, since $U'=U \times_X \xi=U_{\phi(\xi)} \times_{\phi(\xi)} \xi$, and the property of being a split $F$-scheme is invariant by extension of the base field $F$, $U'$ is a split scheme over $k(\xi)$. Therefore $Y' \times_X \xi$ is split, i.e. $f$ has split reduction at $R$.
\end{proof}
\section{Proof of main theorem}
In order to employ Denef's result on arithmetic surjectivity, it suffices to show that some $f'_n$ has geometrically integral generic fiber (Lemma \ref{Lem:gigf}) and has split reduction at every $R \in \DVR(X)$. Indeed, in the language of \cite{DEN}, the fiber of a modification of $f_n'$ over a codimension-$1$ point is split if and only if $f_n'$ has split reduction at its local ring. Establishing this split reduction will be our main focus in this section.

We denote by $\Bad(\pi_i) \subset X, i=1,2$ the set of points $x \in X$ where the morphism $\pi_i:X \rightarrow \P^1$ is not smooth. This coincides with the set of singular points on the singular fibers of $\pi_i$, a finite set of closed points. We endow $\Bad(\pi_i) \subset X$ with the natural reduced scheme structure.

\begin{lemma}\label{Lem:walkaround}
Let $U,U', W, W'$ be $k$-varieties and let
\begin{center}
\begin{equation}
\begin{tikzcd}
W \arrow[d] \arrow[r] & W' \arrow["f",d] \\
U \arrow["\phi",r]  & U' \arrow[phantom, ul, "\ulcorner" near end],
\end{tikzcd}
\end{equation}
\end{center}
be a Cartesian diagram, where $\phi:U \to U'$ is smooth. Let $S$ be a subset of $U'$ and, for each $u \in S$, $L_u$ be a finite field extension of $\kappa(u)$, such that:
\begin{enumerate}[label=(\roman*)]
    \item the fibers of $\phi$ at points $u \notin S$ are split;
    \item for each $u \in S$, the base-changed fiber $\phi^{-1}(u)_{L_u}$ is a split $L_u$-scheme;
    \item for every $w \in W'$ such that $f(w) \in S$, there exists an embedding of $\kappa(f(w))$-field extensions $L_{f(w)}\hookrightarrow \kappa(w)$. 
\end{enumerate}
Then $W \to W'$ is smooth with split fibers. 
\end{lemma}
\begin{proof}
This is immediate. 
\end{proof}

\begin{lemma}\label{Lem:splitoversplit}
Let $k$ be a field, $X$ be a split finite type $k$-scheme and $f:Y \to X$ be a flat morphism with split generic fibers. Then $Y$ is a split $k$-scheme.
\end{lemma}
\begin{proof}
Let $U \subset X$ be a geometrically integral open subscheme, and let $\eta \in U$ be the generic point. Let $V_0$ be a geometrically integral open $\kappa(\eta)$-subscheme of $f^{-1}(\eta)$, and let $V \subset Y$ be an open subscheme such that $V \cap f^{-1}(\eta)=V_0$. Note that $f|_V:V \to U$ is a flat morphism with geometrically integral generic fiber over a geometrically integral base. Hence the base change $V \times_k \overline{k} \to U \times_k \overline{k}$ satisfies the same properties. Applying \cite[Prop.~4.3.8]{LIU} to this last morphism, we deduce that $V$ is geometrically integral. Since $V \subset Y$ is open, $Y$ is split.
\end{proof}

\begin{proposition}\label{Induction1}
Let $n \geq 4$. Assume that there exists a non-empty open subscheme $U \subseteq C_{n-1}$ such that $f_{n-1}|_{U}:U \to X$ is smooth with split fibers. Then, letting $V := f_{n-1}(U)$, there exists an open subscheme $W \subseteq C_n$ such that $f_n|_W:W \to X$ has image $\pi_i^{-1}(\pi_i(V))\setminus \Bad(\pi_i)$ (where $i=1$ if $n$ is odd and $i=2$ if $n$ is even) and is smooth with split fibers.
\end{proposition}
\begin{proof}
After restricting $V$ (resp.~$U$) to $V \setminus \Bad(\pi_i)$ (resp.~ $f_{n-1}^{-1}(V\setminus \Bad(\pi_i)) \cap U$) and noting that $\pi_i^{-1}(\pi_i(V))=\pi_i^{-1}(\pi_i(V\setminus \Bad(\pi_i)))$, we may assume that $\Bad(\pi_i) \cap V = \emptyset$. Therefore, letting $U':= \pi_i(V)\subset \P^1$ (which is open as $\pi_i$ is flat \cite[Exercise III.9.1]{HAR}), we have that $V \to U'$ is smooth.

Let $W' := \pi_i^{-1}(U') \setminus \Bad(\pi_i)$. Note that the composition $U \to V \to U'$ is smooth as it is a composition of smooth morphisms. Let $W := U \times_{U'} W'$. Since $U \rightarrow U'$ is surjective, so is $W \rightarrow W'$. Note that we have a natural open embedding $W = U \times_{U'} W' \subseteq C_{n-1}\times_{\P^1} X = C_n$. 

In order to illustrate the relationship between the various morphisms introduced thus far, we include the following diagram.
\begin{equation}\label{Diagr:PropSplitFibers}
\begin{tikzcd}
                                                                          & C_n \arrow[d, "a_n"] \arrow[rr, "f_n"] &                                                  & X \arrow[d, "\pi_i"']   &                                                      \\
                                                                          & C_{n-1} \arrow[r, "f_{n-1}"]  \arrow[ldd, "\subseteq"', near start, leftarrow]         & X \arrow[r, "\pi_i"]  \arrow[dd,  "\subseteq", near start, leftarrow]  \arrow[ul, phantom, "\ulcorner", very near end, description]                          & \mathbb{P}_K^1  \arrow[rdd, "\subseteq", leftarrow,  near start]               &                                                      \\
W \arrow[d, "a_n|_{W}"'] \arrow[ruu, "\subseteq"] \arrow[rrrr, "f_n|_{W}",near start, crossing over] &                                        &                                                  &                      & W' \arrow[luu, "\subseteq"'] \arrow[d, "\pi_i|_{W'}"] \\
U \arrow[rr, "f_{n-1}|_U"]                        &           \arrow[ul, phantom, "\ulcorner", very near end, description]                              & V  \arrow[rr, "\pi_i|_V"]    &                      & U',                       
\end{tikzcd}
\end{equation}
All of the depth-oriented morphisms are open embeddings and the front and back squares are Cartesian by definition.

We claim that $\restricts{f_n}{W}:W \rightarrow W' \subset X$ has split fibers. To show this, we use Lemma \ref{Lem:walkaround} on the front square of Diagram \eqref{Diagr:PropSplitFibers}. We must define $S$ and the extensions $L_u$, and verify that the assumptions of the lemma hold.

Let $S\subset \P^1_K$ be the (finite) subset over which $\pi_i$ has a singular fiber. For each $u \in S$, let
$\pi_i^{-1}(u)^{reg}$ be the regular locus of $\pi_i^{-1}(u)$ and $\pi_i^{-1}(u)^{reg} \to \Spec L'_u \to \Spec k(u)$ be the Stein decomposition of $\pi_i^{-1}(u)^{reg} \to \Spec k(u)$. Since all fibers of $\pi_i$ are conics, $L'_u$ is a quadratic extension of $k(u)$, possibly split. We define $L_u:= L'_u$ when $L'_u$ is a field, and $L_u:= k(u)$ when $L'_u \cong k(u)^{\oplus 2}$.
Note that all irreducible components of $\pi_i^{-1}(u)^{reg}$ (i.e.\ two affine lines if $L_u=k(u)$ and all of $\pi_i^{-1}(u)^{reg}$ if $L_u$ is a field) are geometrically integral $L_u$-schemes, and that $\pi_i^{-1}(u)$ is a geometrically integral (hence split) $k(u)$-scheme for $u \notin S$.

For each $u \in U'$, let $W'_u$ (resp.~$V_u$, $U_u$) be the fiber of $W' \to U'$ (resp.~$V\to U'$, $U \rightarrow U'$) at $u$.
Note that, for each $u \in S$, $W'_u$ and $V_u$ are open subschemes of $\pi_i^{-1}(u)^{reg}$. In particular:
\begin{enumerate}
    \item $W'_u \to \Spec k(u)$ factors as $W'_u \to \Spec L_u \to \Spec k(u)$ (hence assumption (iii) holds);
    \item $V_u \times_{k(u)} L_u$ is a split $L_u$-scheme (if $L'_u\cong k(u)^{\oplus 2}$ this is clear, otherwise note that $\pi_i^{-1}(u)^{reg}$ is irreducible, hence $V_u$ is dense in it, and $V_u \times_{k(u)} L_u$ is dense in the split $L_u$-scheme $\pi_i^{-1}(u)^{reg} \times_{k(u)} L_u$).
\end{enumerate}

Note that, for each $u \in U'$ (resp.~$u \in S$), the morphism $U_u \to V_u$ (resp.~$U_u\times_{k(u)}L_u \to V_u\times_{k(u)}L_u$) is surjective, smooth and has split fibers, as all of these properties are invariant under base change and they are satisfied by the morphism $U \to V$.

Applying, for each $u \notin S$ (resp.~$u \in S$), Lemma \ref{Lem:splitoversplit} to the morphism $U_u \to V_u$ (resp.~$U_u\times_{k(u)}L_u \to V_u\times_{k(u)}L_u$), we deduce that $U_u$ (resp.~$U_u\times_{k(u)}L_u$) is a split $k(u)$ (resp.~$L_u$)-scheme, i.e.\ assumptions (i) and (ii) hold.

We deduce from Lemma \ref{Lem:walkaround} that the morphism $W \to W'$ is smooth with split fibers, hence we have proved our claim.

Remembering that $W'=\pi_i^{-1}(\pi_i(V))\setminus \Bad(\pi_i)$ and noting that $W$ is an open subscheme of $C_n$, this proves the proposition.
\end{proof}

\begin{proposition}\label{Induction2}
Let $n \geq 4$. Let $R \in \DVR(X)$ be such that $f'_{n-1}:C'_{n-1} \rightarrow X$ has split reduction at $R$. Then $f'_n:C'_n \rightarrow X$ has split reduction at $R$. 
\end{proposition}
\begin{proof}
Recall that Diagram \eqref{DiagramCn} implies the existence of a section $\sigma_n:C_{n-1} \rightarrow C_n$ (commuting with projection to $X$) to the morphism $a_n:C_n \rightarrow C_{n-1}$. 
Being the dominant base change of a generically smooth morphism, $a_n$ is generically smooth, i.e.\ there exists an open subscheme $U \subseteq C_{n-1}$ such that $a_n^{-1}(U)$ is a smooth open subscheme of $C_n$. 
In particular, the image $\sigma_n(\eta(C_{n-1}))$ of the generic point $\eta(C_{n-1})$ of $C_{n-1}$ is a smooth point of the $K$-variety $C_n$. Therefore, if $C'_{n-1} \rightarrow C_{n-1}$ is a desingularization of $C_{n-1}$ and $C'_n \rightarrow C_n$ is a desingularization of $C_n$, then there exists a rational map $\sigma'_n:C'_{n-1} \dashrightarrow C'_n$, commuting with projection to $X$. Denoting by $(C'_{n-1})_{K(X)}$ and $(C'_{n})_{K(X)}$ the generic fibers of $f'_{n-1}$ and $f'_n$, we have that $\sigma'_n$ induces a rational map $(C'_{n-1})_{K(X)}\dashrightarrow (C'_{n})_{K(X)}$ of $K(X)$-varieties. Since $(C'_{n-1})_{K(X)}$ has split reduction at $R$ (see Definition \ref{def:splitred}), by \cite[Lem.~1.1(a)]{SKO} there exists a residually closed local flat extension of DVRs $i:R \to R'$ such that, denoting by $K(R')$ the fraction field of $R'$, $(C'_{n-1})_{K(X)}$ has a $K(R')$-point. By Nishimura's lemma and the existence of a rational map $(C'_{n-1})_{K(X)}\dashrightarrow (C'_{n})_{K(X)}$, we deduce that $(C'_{n})_{K(X)}$ has a $K(R')$-point. This last condition implies, by \cite[Lem.~1.1(b)]{SKO}, that $(C'_{n})_{K(X)}$ has split reduction at $R$, thus concluding the proof of the proposition.
\end{proof}

\begin{corollary}\label{Cor:EverywhereSplit}
If $\Bad(\pi_1) \cap \Bad(\pi_2)=\emptyset$, then, for every $R \in \DVR(X)$, the morphism $f'_5:C'_5 \rightarrow X$ has split reduction at $R$.
\end{corollary}
\begin{proof}
By Lemma \ref{Lem:gigf}, the generic fiber of $\pi_1\circ f_2:C_2 \to \P^1$ is geometrically integral. Hence, for a sufficiently small neighbourhood $V$ of the generic point of $C_2$, we may assume that $\pi_1\circ f_2|_V:V \to \P^1$ is smooth with geometrically integral fibers. We let $U :=\pi_1( f_2(V))$.

We define $U_n$ by $U_3 = \pi_1^{-1}(U)$ and $U_n = \pi_i^{-1}(\pi_i(U_{n-1}))\setminus \Bad(\pi_i)$ (where $i=1$ if $n$ is odd and $i=2$ if $n$ is even) for $n=4,5$.

Note that $\pi_1^{-1}(U)$ contains the generic fiber $\pi_1^{-1}(\eta\left(\P^1\right)),$ where $\eta\left(\P^1\right) \in \P^1$ denotes the generic point. In particular, $\pi_2(U_3)=\pi_2(\pi_1^{-1}(\eta\left(\P^1\right)))=\P^1$, hence $U_4= X \setminus \Bad (\pi_2)$. Analogously, $U_5= X \setminus \Bad (\pi_1)$.

We claim that, for $n=3,4,5$, there exists an open subscheme $V_n \subseteq C_n$ such that $f_n(V_n)=U_n$ and $f_n|_{V_n}:V_n \to U_n$ is smooth with split fibers. Indeed, for $n=3$, remembering that the properties of being smooth and having split fibers are invariant under base change, this holds with $V_3:=V \times_{\P^1,\pi_1} X \subseteq C_3$; while for $n=4,5$ this follows (by induction) applying Proposition \ref{Induction1} with $U=V_{n-1}$ and letting $V_n:=W$.

Let, for $i=1,2$, $X_i := X \setminus \Bad (\pi_i)$.
We deduce from Proposition \ref{Prop:SplitImpliesSplit} applied to $f=f_4$, (resp.~$f=f_5$) and $U=V_4$ (resp.~$U=V_5$), that $f'_4$ (resp.~$f'_5$) has split reduction for all $R \in \DVR(X_1)$ (resp.~for all $R \in \DVR(X_2)$). Since $f'_4$ has split reduction for all $R \in \DVR(X_1)$, it follows by Proposition \ref{Induction2} that the same holds for $f'_5$. 

The assumption that $\Bad(\pi_1) \cap \Bad(\pi_2)=\emptyset$ implies that $\DVR(X_1) \cup \DVR(X_2)=\DVR(X)$. Therefore, we deduce that $f'_5:C'_5 \rightarrow X$, which has split reduction for all $R \in \DVR '(X_1)\cup \DVR(X_2)$, has split reduction for all $R \in \DVR(X)$, as wished.
\end{proof}

We now prove the following proposition, the heart of the proof of Theorem  \ref{mainthm}.

\begin{proposition}\label{MainProp}
If $\Bad(\pi_1) \cap \Bad(\pi_2)=\emptyset$, then $f'_5:C'_5 \rightarrow X$ is arithmetically surjective. That is, $f'_5\left(C'_5\left(K_v\right)\right) = X\left(K_v\right)$ for all but finitely many places $v \in \Val\left(K\right)$.
\end{proposition}
\begin{proof}
The proposition is an application of \cite[Thm.~1.2]{DEN} to the morphism $f'_5:C'_5 \rightarrow X$. We need only verify that $f'_5$ has geometrically integral generic fiber (Lemma \ref{Lem:gigf}) and that, using the terminology of \cite{DEN}, for every birational modification $\widetilde{f'_5}:\widetilde{C'_5} \rightarrow \widetilde{X}$ of ${f'_5}:{C'_5} \rightarrow {X}$ and every divisor $D \in \widetilde{X}^{(1)}$, the fiber $\widetilde{f'_5}^{-1}(D)$ is split. With our definitions, this last condition (that $\widetilde{f'_5}^{-1}(D)$ is split) means precisely that $f'_5:C'_5 \rightarrow X$ has split reduction at the DVR $\mathcal{O}_{\widetilde{X},D} \subset \DVR(X)$, which follows from Corollary \ref{Cor:EverywhereSplit}.
\end{proof}

\begin{proof}[Proof of Theorem \ref{mainthm}]
Let $\Val\left(K\right)$ be the set of places of $K$ and let $S \subset \Val\left(K\right)$ be a finite subset such that $f'_5(C'_5(K_v)) \to X(K_v)$ is surjective for all $v \notin S$. We will show that $X$ satisfies weak approximation off $S$. Let $(P_v)_{v \notin S} \in \prod_{v \notin S} X(K_v)$ be a collection of local points. By Proposition \ref{MainProp}, there exists $(\widetilde{P}_v)_{v \notin S} \in \prod_{v \notin S} C'_5(K_v)$ such that $f'_5(\widetilde{P}_v)=P_v$ for all $v\notin S$. 

As noted in Section \ref{Sec3}, $C'_n$ satisfies weak approximation for every $n \geq 0$. In particular, there exists $\widetilde{P} \in C'_5(K)$ that is arbitrarily close (in $\prod_{v \notin S} C'_5(K_v)$) to $(\widetilde{P}_v)_{v \notin S}$, hence $f'_5(\widetilde{P}) \in X\left(K\right)$ is arbitrarily close to $(P_v)_{v \notin S}$.
\end{proof}

\begin{remark}
A natural question arising from Proposition \ref{MainProp} is whether $5$ is the minimum value of $n$ for which $C'_n \to X$ arithmetically surjective. We shall therefore give some indication of the picture for $n \leq 4$.
\begin{enumerate}
    \item[$n\leq 1$:] The morphism $C'_n \to X$ is not surjective, so it is certainly not arithmetically surjective.
    \item[$n=2$:] The morphism $C'_n \to X$ is not arithmetically surjective except in the trivial case where the morphism $(\pi_1,\pi_2):X \to \P^1 \times \P^1$ is birational. The morphism $C'_2 \to X$ is generically finite of degree $\pi_1^{-1}(t_1)\cdot \pi_2^{-1}(t_2)$ (which is also the degree of $(\pi_1,\pi_2)$), and so cannot be arithmetically surjective when this intersection number is greater than $1$ (see \cite[Prop.~3.5.2]{SER}).
    \item[$n=3$:] The morphism $C'_n \to X$ is not arithmetically surjective in the following scenario: $\pi_1$ has a non-split singular fiber with singular point $P$, and the finite $K$-scheme $\pi_2^{-1}(\pi_2(P)) \cap C_1$ contains no points defined over the residue field $\kappa(P)$ (this condition is satisfied if the $K$-point $\pi_1(P_0)$ lies outside the thin subset $\pi_1(\pi_2^{-1}(\pi_2(P)))(\kappa(P)) \subseteq \P^1(\kappa(P))$). Indeed, let $\widetilde{X}$ be the blowup of $X$ at $P$ and let $E \subseteq \widetilde{X}$ be the exceptional divisor. Let $R \subseteq K(\widetilde{X})=K(X)$ be the DVR associated to the generic point of the irreducible divisor $E$. Then one can show that $C'_3 \to X$ {\bfseries does not} have split reduction at $R$.
    We omit the proof of this fact, which is obtained through a calculation similar to the one given in Section \ref{Sec:GBsp} of this paper (mimicking, in particular, the decomposition \ref{Decomposition}).
   
    \item[$n=4$:] The morphism $C'_n \to X$ is in some sense ``quite often'' arithmetically surjective. One can prove, using essentially the same proof as that of Corollary \ref{Cor:EverywhereSplit}, that $C'_4 \to X$ is arithmetically surjective if $P_0$ lies on a smooth fiber of $\pi_2$ and is such that, for every $P \in \Bad(\pi_2)$:
    \begin{itemize}
        \item if $\pi_1^{-1}(\pi_1(P))$ is $\kappa(P)$-split, then so is $\pi_1^{-1}(\pi_1(P)) \times_{\pi_2, \P^1,\pi_2} C_1$;
        \item if $\pi_1^{-1}(\pi_1(P))$ splits after a quadratic extension $L(P)/\kappa(P)$, then the fiber product $\pi_1^{-1}(\pi_1(P)) \times_{\pi_2, \P^1,\pi_2} C_1$ splits over $L(P)$.
    \end{itemize}
    Using the ``disjoint branch loci lemma'' \cite[Lem.~2.8]{STR} and the fact that the branch locus of the morphism $\pi_2|_{C_1}:C_1=\pi_1^{-1}(\pi_1(P_0)) \to \P^1$ contains no fixed points as $P_0$ varies (see the proof of Lemma \ref{Lem:gigf}) one may show that the last two conditions are satisfied if $P_0$ is Zariski-generic enough (same for the first, but this holds trivially).
    
    Again, we omit the proof that the conditions above are enough to deduce that $C'_4 \to X$ is arithmetically surjective.
    
The question of whether there exists an example where $C_4' \rightarrow X$ is not arithmetically surjective for any choice of $P_0$ remains open.
\end{enumerate}

\end{remark}

\begin{remark} \label{GBsp}
The assumption $\Bad(\pi_1) \cap \Bad(\pi_2)=\emptyset$ is not satisfied by all double conic bundles on del Pezzo surfaces of degrees $1$ or $2$. For example, consider the del Pezzo surface $X$ of degree $2$ given by
$$
X: w^2 = x^4 + 4x^2 y^2 + z^4 \subset \mathbb{P}\left(1,1,1,2\right).
$$
Note that, after rearranging the equation, we may factorise both sides to obtain
$$
X: g_1 g_2 = h_1 h_2,
$$
where
$$
\begin{aligned}
g_1\left(x,y,z,w\right) = w - \left(x^2 + 2y^2\right), \\
g_2\left(x,y,z,w\right) = w + \left(x^2 + 2y^2\right), \\
h_1\left(x,y,z,w\right) = z^2 - 2y^2, \\
h_2\left(x,y,z,w\right) = z^2 + 2y^2.
\end{aligned}
$$
From this factorisation we obtain the (dual) conic bundles
$$
\begin{aligned}
\pi_1: X \rightarrow \mathbb{P}^1, \quad [x:y:z:w] \mapsto [g_1:h_1] \quad \left(\textrm{or $[h_2:g_2]$}\right), \\
\pi_2: X \rightarrow \mathbb{P}^1, \quad [x:y:z:w] \mapsto [-g_2:h_1] \quad \left(\textrm{or $[-h_2:g_1]$}\right).
\end{aligned}
$$
It is easily seen that the points $[1,0,0,\pm 1]$ belong to $\Bad(\pi_1) \cap \Bad(\pi_2)$. Further, none of the exceptional meeting at these points is rational, so one cannot simply blow one of them down to obtain a cubic surface.

On the other hand, a general del Pezzo surface of degree $2$ possesses no point at which four exceptional curves meet.
Indeed, such a point exists if and only if the plane quartic over which the anticanonical model ramifies has an involution, and the generic plane quartic has trivial automorphism group. Consequently, one may think of examples with $\Bad(\pi_1) \cap \Bad(\pi_2) \neq \emptyset$ as rare.
\end{remark}

\begin{proof}[Proof of Corollary \ref{cor:dp4}]
Since $X$ contains a rational point, it is unirational, hence $X\left(K\right)$ is dense. Blowing up any two points $P_1,P_2 \in X\left(K\right)$ not on exceptional curves and not both on a curve of self-intersection zero, we obtain a del Pezzo surface $Z$ of degree $2$. Note that, blowing up first at $P_1$, we obtain a del Pezzo surface $Y_1$ of degree $3$ with a $K$-rational exceptional curve $E_1$. The class $-K_{Y_1} - E_1$ gives rise to a conic fibration on $Y_1$, and the pullback of this class gives rise to a conic fibration with class $C_1$ on $Z$. Similarly, we may first blow up at $P_2$ to produce a surface $Y_2$ with exceptional divisor $E_2$ and pull back the class $-K_{Y_2} - E_2$ to a class $C_2$ on $Z$ giving a conic fibration. It follows from \cite[Prop.~V.3.2]{HAR} that
$$
C_1 = - K_Z - L_1 + L_2, \quad C_2 = - K_Z + L_1 - L_2,
$$
where $L_1$ and $L_2$ are the pullbacks of $E_1$ and $E_2$ respectively. 

In order to apply Theorem \ref{mainthm}, it remains to show that we can choose $P_1$ and $P_2$ so that the singular fibers of these conic fibrations on $Z$ share no singular point.

Suppose that there exists a bad point $Q$ on $Z$ (i.e.\ a shared singular point of fibers from the two conic fibrations), and let $M_{i,j}$, $i,j \in \{1,2\}$ be the exceptional curves meeting at $Q$ so that $C_i = M_{i,1} + M_{i,2}$. It follows from elementary intersection multiplicity calculations and the adjunction formula that the projection of each $M_{i.j}$ onto $X$ gives a class $D_{i,j}$ whose smooth members are conics on $X$. It is well-known that there are ten one-dimensional families of conics on $X$ (defined over $\overline{K}$), and that these split into pairs whose sum (as classes) gives $-K_X$. It follows that $D_{i,1} + D_{i,2} = -K_X$. Since $M_{i,j} \cdot L_i = 2$ and $M_{i,j} \cdot L_{3-i} = 0$, it follows that $D_{i,1}$ and $D_{i,2}$ are two conics on $X$ meeting precisely at $P_i$ and $Q$. So, it suffices to show that we can choose $P_1$ and $P_2$ so that, picking any of the five dual pairs of conics through $P_1$ and any of the four remaining four dual pairs of conics through $P_2$, these four conics do not all meet in one point.

Choose $P_1 \in X\left(K\right)$ to be any rational point not on an exceptional curve of $X$. Denote by $\left(F_i,-K_X-F_i\right)$, $i=1,\dots,5$ the five pairs of dual conic classes on $X$. For each pair, consider the unique pair of representative curves through $P_1$. Note that $F_i \cdot \left(-K_X - F_i\right) = 2$, hence these representative curves intersect in at most one other point $Q_i$. Now consider the unique pair of dual conics from the four remaining pairs through each $Q_i$. Again, each such pair intersects in at most one other point $R_{i,j}$. There are at most twenty points of the form $R_{i,j}$. Choosing $P_2$ outside of this finite set and not on any exceptional curve of $X$ or the ten conics containing $P_1$, we deduce that the surface $Z_{P_1,P_2}$ obtained by blowing up $X$ at $P_1$ and $P_2$ is a del Pezzo surface of degree $2$ satisfying the hypotheses of Theorem \ref{mainthm}, hence it satisfies weak weak approximation, and therefore $X$ also satisfies weak weak approximation.
\end{proof}

\section{Intersecting bad loci}\label{Sec:GBsp}

In this section, we explain why our method fails when condition $(b)$ in Theorem \ref{mainthm} does not hold. This is encapsulated in the following result.
\begin{proposition}\label{Conditionbdoesnothold}
Let $X$ be a smooth projective surface over a number field $K$ endowed with two conic fibrations $\pi_i: X \rightarrow \mathbb{P}^1$, $i=1,2$, satisfying condition (a) of Theorem \ref{mainthm}. Suppose that there exists $P_0 \in X\left(K\right)$ lying on a smooth fiber of $\pi_1$, and for each $n \geq 0$, let $f'_n:C'_n \to X$ be the morphism defined in Section \ref{Sec3}. Assume that there exists a point $P \in X(K)$ such that:
\begin{enumerate}
\item $P$ lies in the intersection $\Bad(\pi_1) \cap \Bad(\pi_2)$, and
\item $\pi_1^{-1}(\pi_1(P))$ and $\pi_2^{-1}(\pi_2(P))$ are not split.
\end{enumerate}
Then, for every $n \geq 3$, $f'_n$ is not arithmetically surjective.
\end{proposition}

Note that the assumptions of Proposition \ref{Conditionbdoesnothold} are satisfied by the example given in Remark \ref{GBsp}.

\begin{remark}\label{RmkPnotoverK}
The proof may be generalized to the case where $P$ is not defined over $K$, assuming that $\pi_1^{-1}(\pi_1(P))$ and $\pi_2^{-1}(\pi_2(P))$ are non-split over the field of definition of $P$. However, since our primary purpose is to show that our method fails to produce arithmetically surjective covers in certain scenarios, we assume that $P$ is $K$-rational for the sake of simplicity.
\end{remark}

\begin{remark}\label{RmkOnlyGreaterThan4IsInteresting}
Proposition \ref{Conditionbdoesnothold} holds for $n=0,1,2$ as well, but it is not interesting. Indeed, for $n=0,1$, the morphism $f'_n$ is not dominant, hence it cannot be arithmetically surjective. For $n=2$, the morphism $f'_2$ is finite of degree greater than $1$.
In fact, $f'_2$ has degree $(\pi_1^{-1}(A) . \pi_2^{-1}(B))$, which is greater than or equal to the local intersection multiplicity $(\pi_1^{-1}(\pi_1(P)) . \pi_2^{-1}(\pi_2(P)))_P$, which is at least $4$.
\end{remark}

To prove Proposition \ref{Conditionbdoesnothold} we will use \cite[Thm.~1.4]{LSS}, which provides a converse to Denef's result via the introduction of \emph{pseudo-split} schemes, the definition of which we give below.

\begin{mydef}
Let $k$ be a perfect field with algebraic closure $\overline{k}$. A $k$-scheme $X$ is \emph{pseudo-split} if for every element $\gamma$ of the absolute Galois group $\Gamma_k= \Gal\left(\overline{k}/k\right)$, there exists a multiplicity-$1$ irreducible component of $X \times_k \overline{k}$ fixed by $\gamma$.
\end{mydef}

Note that every split scheme is pseudo-split. Indeed, a scheme $X$ as in the above definition is split if and only if there exists a geometrically integral multiplicity-$1$ irreducible component, which is then necessarily fixed by every $\gamma \in \Gamma_k$. 

Denef proves in \cite{DEN} that a dominant morphism $f:Y \to X$ of smooth projective geometrically integral $K$-varieties with geometrically integral fiber is arithmetically surjective \emph{if}, for every birational modification $\widetilde{f}:\widetilde{Y}\to \widetilde{X}$ of $f$, the morphism $\widetilde{f}$ has split fibers over all codimension-$1$ points of $\widetilde{X}$. Loughran, Skorobogatov and Smeets later proved in \cite{LSS} that the \emph{if} of Denef's result may be replaced by \emph{if and only if} upon replacing ``split'' by ``pseudo-split''. So, to prove that (for every $n \geq 3$) $f'_n$ is not arithmetically surjective, it suffices to provide a birational modification $\widetilde{f'_n}:\widetilde{C'_n} \to \widetilde{X}$ of $f'_n$ and an irreducible divisor $D \subseteq \widetilde{X}$ such that $\widetilde{f'_n}$ does not have a pseudo-split fiber above the generic point of $D$.

In particular, using \cite[Thm.~1.4]{LSS}, Proposition \ref{Conditionbdoesnothold} follows immediately from the following lemma.

\begin{lemma}\label{MainLemmaCounterexample}
 Let $X$ and $P$ be as in Proposition \ref{Conditionbdoesnothold}. Let $\nu:\widetilde{X} \to X$ be the blowup of $X$ at $P$, and let $E \subseteq \widetilde{X}$ be the exceptional divisor $\nu^{-1}(P)$. If $\widetilde{f'_n}:\widetilde{C'_n} \to \widetilde{X}$ is any birational modification of $f'_n$ whose codomain is the blowup $\widetilde{X}$, then the fiber of $\widetilde{f'_n}$ above the generic point of $E$ is not pseudo-split.
\end{lemma}

\begin{remark}
As with Proposition \ref{Conditionbdoesnothold}, Lemma \ref{MainLemmaCounterexample} may be generalized to the case $P \not\in X\left(K\right)$. In this case, one should instead take the simultaneous blowup of all conjugates of $P$, i.e.\ the blowup at the closed $K$-point associated to $P$.
\end{remark}

\subsection{Proof of Lemma \ref{MainLemmaCounterexample}}

Note that, by \cite[Prop.~2.12]{LSS}, it suffices to prove the claim for one modification. We start by constructing a modification which is particularly easy to describe.

We denote by $\widetilde{\pi_i}: \widetilde{X} \to \P^1$ the composition $\pi_i\circ \nu$ for $i=1,2$. We warn the reader that these morphisms are no longer conic fibrations, as the fiber of $\widetilde{\pi_i}$ at $\pi_i\left(P\right)$ has three geometrically connected components (two coming as proper transforms from the ones on $X$ plus the exceptional divisor).

Note that, in Definition \ref{DefCn}, we never used the hypothesis $\pi_1$ and $\pi_2$ are conic fibrations. Since this is the only hypothesis from that setting that does not hold for the $\widetilde{\pi_i}$, we may define varieties $\widetilde{C}_n$ as in Definition \ref{DefCn} but with $\widetilde{\pi_i}$ and $X$ replacing $\pi_i$ and $X$ respectively. Moreover, note that Proposition \ref{PropGeomint} made use only of the fact that the geometric generic fiber of each $\pi_i$ is rational, which also holds for $\widetilde{\pi_i}$. We deduce that, for each $n \geq 1$, $\widetilde{C}_n$ is a geometrically integral variety of dimension $n$. Finally, note that we may inductively define natural morphisms $\widetilde{C}_n \to C_n$, and that these morphisms are birational. In analogy with the notation of Section \ref{Sec3}, we define $\widetilde{C}'_n$ to be a desingularization of $\widetilde{C}_n$. As a visual aid for the reader (and the authors), let us draw the following commutative diagram ($n \geq 0$):
\[
\begin{tikzcd}
\widetilde{f}_n': & \widetilde{C}_n' \arrow[d, "\text{bir}", dashed] \arrow[r, "\text{desing}"] & \widetilde{C}_n \arrow[r, "\widetilde{f}_n"] \arrow[d, "\text{bir}"] & \widetilde{X} \arrow[d, "\text{blowup at } p"] \\
f'_n:             & {C}_n' \arrow[r, "\text{desing}"]                                           & {C}_n \arrow[r, "f_n"]                                               & X                                             
\end{tikzcd}
\]
where ``bir'' stands for birational and ``desing'' for desingularization, and $\widetilde{f}_n'$ is defined as the composition on the first row. As is clear from the diagram, the morphism $\widetilde{C}_n' \to \widetilde{X}$ is a birational modification of $C'_n \to X$. In particular, by \cite[Prop.~2.12]{LSS}, to prove the lemma it suffices to show that the fiber $(\widetilde{f}_n')^{-1}(\eta)$ is not pseudo-split, where $\eta$ denotes the generic point of the exceptional divisor $E$ of $\widetilde{X} \to X$.

To prove that $(\widetilde{f}_n')^{-1}(\eta)$ is not pseudo-split, we are actually going to prove that it does not contain \emph{any} pseudo-split $k(\eta)$-subscheme. Although this looks more difficult, we have the following lemma.

\begin{lemma}\label{LemUinUQS}
Let $f:X \to Y$ be a morphism of $k$-schemes. Assume that $X$ has a pseudo-split $k$-subscheme, then so does $Y$.
\end{lemma}
\begin{proof}
We may assume, without loss of generality, that $X$ is pseudo-split and $f$ is dominant. Since $X$ is pseudo-split, so is $X^{sm}$, by definition, so we may also assume without loss of generality that $X$ is smooth. Moreover, noticing that $X \to Y$ factors through $Y_{red}$ (since $X$ is smooth), after substituting $Y$ with its reduced subscheme $Y_{red}$, we may assume that it is reduced. Finally, substituting $Y$ with an open subscheme, we may assume that it is smooth as well (recall that $\Char k=0$). Hence, the morphism $X \to Y$ induces now a morphism $\Spec A_X \to \Spec A_Y$ (for the definition of the $k$-finite \'etale algebras $A_X$ and $A_Y$, see \cite[\S2.2]{LSS}). According to \cite[Defs.~2.2,~2.10]{LSS}, that $X$ is pseudo-split means that each element of $\Gamma_k$ fixes at least one point of $\Spec A_X(\overline{k})$. The existence of a $k$-morphism $\Spec A_X \to \Spec A_Y$ then immediately implies that the same holds for $\Spec A_Y$, i.e., $Y$ is pseudo-split.
\end{proof}

\begin{remark}
Note that Lemma \ref{LemUinUQS} is false if one replaces the property of ``having a pseudo-split $k$-subscheme'' with the property of ``being pseudo-split'', as the following example shows: take $f$ to be the embedding of the unique $k$-point $P$ of a non-split reduced singular conic $C$ (or, in other words, two geometric lines in the plane switched by Galois action), e.g.\ $P = [0,0,1]$, $C: x^2 - 2y^2 = 0 \subset \mathbb{P}^2$ over $k = \mathbb{Q}$.
\end{remark}

Note that there is an obvious $k(\eta)$-morphism $(\widetilde{f}_n')^{-1}(\eta) \to \widetilde{f}_n^{-1}(\eta)$. Hence, by Lemma \ref{LemUinUQS}, we see that to prove that 
$(\widetilde{f}_n')^{-1}(\eta)$ is not $k(\eta)$-pseudo-split, it is enough to prove  the following lemma, which concludes the proof of Lemma \ref{MainLemmaCounterexample}.
\begin{lemma}\label{LemNotcontain}
The $k(\eta)$-scheme $\widetilde{f}_n^{-1}(\eta)$ does not contain any $k(\eta)$-pseudo-split subscheme.
\end{lemma}

The advantage afforded to us by the preceding lemma is that we may hope to describe the fiber $\widetilde{f}_n^{-1}(\eta)$ explicitly. In fact, this is essentially what shall be done in the proof of Lemma \ref{LemNotcontain} (we actually give an explicit description of $\widetilde{f}_n^{-1}(E)$).

\begin{proof}[Proof of Lemma \ref{LemNotcontain}]

Recall from Remark \ref{RmkExplicitCn}, that we have:
\begin{equation}\label{Cn}
\widetilde{C_n} =
    \begin{cases}
 \overbrace{\widetilde{X} \times_{\widetilde{\pi}_1,\widetilde{\pi}_1} \widetilde{X} \times_{\widetilde{\pi}_2,\widetilde{\pi}_2} \cdots \times_{\widetilde{\pi}_2,\widetilde{\pi}_2}\widetilde{X} }^{n} \times_{\widetilde{\pi}_1,f_0\circ \widetilde{\pi}_1} P_0 \ & \text{ if } n \text{ is odd }; \\
 \overbrace{\widetilde{X} \times_{\widetilde{\pi}_2,\widetilde{\pi}_2} \widetilde{X} \times_{\widetilde{\pi}_1,\widetilde{\pi}_1} \cdots \times_{\widetilde{\pi}_2,\widetilde{\pi}_2} \widetilde{X}}^n \times_{\widetilde{\pi}_1,f_0\circ \widetilde{\pi}_1} P_0 \ & \text{ if } n \text{ is even},
\end{cases}
\end{equation}

and that, under this identification, $\widetilde{f}_n:\widetilde{C_n} \to \widetilde{X}$ corresponds to the projection on the first factor.

From now on let us assume that $n$ is odd so that we fix the formula to use (the proof for even $n$ is identical), and let us fix such an $n$.

Let us write:
\[
\widetilde{C_n} = {\widetilde{X} \times_{\P^1} \widetilde{X} \times_{\P^1} \cdots \times_{\P^1} \widetilde{X} } \times_{\P^1} P_0,
\]
leaving it implicit that the morphisms that define the fibered products are those of \eqref{Cn}, and that the number of copies of $\widetilde{X}$ is $n$. 

For $i \in \{1,2\}$, define $D_i$ to be the strict transform of $\pi_i^{-1}(\pi_i(P))$. Note that $\widetilde{\pi}_i^{-1}(\pi_i(P))_{red}=D_i \cup E$. \footnote{Although this is irrelevant for our proof, as divisors, we have that $\widetilde{\pi}_i^{-1}(\pi_{i}(P))=D_i + 2E$. This is why we are about to choose to work with reduced schemes everywhere, so that we can get rid of the nilpotents at $E$, which we do not need to keep track of.} 

Since we only need to work with reduced schemes (indeed, any pseudo-split subscheme of $\widetilde{f}_n^{-1}(\eta)$ would necessarily be contained in $\widetilde{f}_n^{-1}(\eta)_{red}$), we make the following:

\begin{UnNot}
From now on in this proof, we will only work with reduced schemes. So, to simplify certain technical matters, all schemes will be identified with their associated reduced scheme. In particular, for a morphism $f:X \to Y$ and a subscheme $Z \subseteq Y$, $f^{-1}(Z)$ will actually denote the scheme $f^{-1}(Z)_{red}$. Analogously, for morphisms $X \to Z, Y \to Z$, the notation $X \times_Z Y$ will actually denote the scheme $(X \times_Z Y)_{red}$.
\end{UnNot}

As mentioned above, we will give an explicit expression for the fiber $\widetilde{f}_n^{-1}(\eta)$.  We claim that there is a ``decomposition'':
\begin{equation}\label{Decomposition}
    \begin{split}
     \widetilde{f}_n^{-1}(E)  = & E \times_{t_1} D_1 \times_{\P^1}  \widetilde{X} \times_{\P^1} \cdots \widetilde{X} \times_{\P^1} P_0 \ \cup \\
     &  E \times_{t_1} E \times_{t_2} D_2 \times_{\P^1} \cdots \widetilde{X} \times_{\P^1} P_0 \ \cup  \\
     & \cdots \ \cup  \\
    &E \times_{t_1} E \times_{t_2} E \times_{t_1} \cdots D_2 \times_{\P^1} P_0 \\
     \subseteq & {\widetilde{X} \times_{\P^1} \widetilde{X} \times_{\P^1} \cdots \times_{\P^1} \widetilde{X} } \times_{\P^1} P_0,
\end{split}
\end{equation}
where $t_1:=\pi_1(P), t_2:= \pi_2(P) \in \P^1(K)$.

To prove the claim, note that, by definition:
\begin{align*}
    \widetilde{f}_n^{-1}(E) = & E\times_{\P^1} \widetilde{X} \times_{\P^1} \cdots \times_{\P^1} \widetilde{X}  \times_{\P^1} P_0 \\
     \subseteq & {\widetilde{X} \times_{\P^1} \widetilde{X} \times_{\P^1} \cdots \times_{\P^1} \widetilde{X} } \times_{\P^1} P_0,
\end{align*}
However, the image of $E$ in $\P^1$ under $\widetilde{\pi_1}|_{E}$ (the first implied morphism on the left) is just the point $t_1 \cong \Spec k$, i.e., we have the following factorization:
\[
\begin{tikzcd}
E \arrow[r] \arrow[rr, "{\widetilde{\pi_1}|_E}", bend left] & t_1 \arrow[r] & {\P^1}
\end{tikzcd}.
\]
Hence, we have that $E\times_{\P^1} \widetilde{X} = E \times_{t_1} (t_1 \times_{\P^1} \widetilde{X}) = E\times_{t_1} \widetilde{\pi}_1^{-1}(t_1)$. Moreover,
remembering the fact that $\widetilde{\pi}_{1}^{-1}(t_1) = D_{1} \cup E$, we deduce that $E\times_{\P^1} \widetilde{X}  = E\times_{t_1} \widetilde{\pi}_1^{-1}(t_1) =E\times_{t_1} \widetilde{\pi}_1^{-1}(t_1)  =  (E\times_{t_1} D_1) \cup (E\times_{t_1} E)$.

Therefore we obtain:
\begin{align*}
    \widetilde{f}_n^{-1}(E)  = &E\times_{\P^1} \widetilde{X} \times_{\P^1} \cdots \times_{\P^1} \widetilde{X}  \times_{\P^1} P_0 \\
    = & E \times_{t_1} D_1 \times_{\P^1}  \widetilde{X} \times_{\P^1} \cdots \widetilde{X} \times_{\P^1} P_0  \ \cup\\
    &  E \times_{t_1} E \times_{\P^1} \widetilde{X} \times_{\P^1} \cdots \widetilde{X} \times_{\P^1} P_0  \\
    & \subseteq {\widetilde{X} \times_{\P^1} \widetilde{X} \times_{\P^1} \cdots \times_{\P^1} \widetilde{X} } \times_{\P^1} P_0,
\end{align*}
We may apply the same argument as before to deduce that $(E \times_{\widetilde{\pi_2},\widetilde{\pi_2}} \widetilde{X} ) = (E \times_{t_2} D_2) \cup (E\times_{t_2} E)$, and then, with an easy induction, deduce that:
\begin{align*}
    \widetilde{f}_n^{-1}(E)  = & E \times_{t_1} D_1 \times_{\P^1}  \widetilde{X} \times_{\P^1} \cdots \widetilde{X} \times_{\P^1} P_0  \ \cup\\
    &  E \times_{t_1} E \times_{t_2} D_2 \times_{\P^1} \cdots \widetilde{X} \times_{\P^1} P_0  \ \cup  \\
    & \cdots \ \cup  \\
    & E \times_{t_1} E \times_{t_2} E \times_{t_1} \cdots D_2 \times_{\P^1} P_0  \ \cup\\
    & E \times_{t_1} E \times_{t_2} E \times_{t_1} \cdots E \times_{\P^1} P_0 \\
    & \subseteq {\widetilde{X} \times_{\P^1} \widetilde{X} \times_{\P^1} \cdots \times_{\P^1} \widetilde{X} } \times_{\P^1} P_0.
\end{align*}
To prove the claim it suffices to prove now that $E \times_{t_1} E \times_{t_2} E \times_{t_1} \cdots E \times_{\P^1} P_0 = \emptyset$. However, this is immediate, because the image of $E$ in $\P^1$ under $\widetilde{\pi_1}|_{E}$ (the penultimate implied morphism from the right) is $t_1$, while the image of $P_0$ under ${\pi_1}|_{P_0}$ (the last implied morphism on the right) is $\pi_1(P_0)$ which is different from $t_1$, because $P_0 \in X(K)$ is different from $P$, and it cannot be a smooth point of $\pi_1^{-1}(t_1)=\pi_1^{-1}(\pi_1(P))$, because otherwise this fiber would be split over $K$, which it is not by hypothesis. This proves the claim.

\begin{remark}
The fact that $P_0$ is different from $P$ is not an arbitrary choice. Note that, if it were chosen to be $P$, then for each $n \geq 1$, the $k$-scheme $C_n$ formed as in Definition \ref{DefCn}, would be an integral, but not geometrically irreducible variety\footnote{The integrality follows from the fact that $C_1=\pi_1^{-1}(\pi_1(P))$ is integral, and then an induction argument, similar to the one of Lemma \ref{Lem:gigf}, can be used to show that $C_n$ is integral for all $n$. On the other hand, it cannot be geometrically irreducible as there is a dominant morphism $C_n \to C_1$, and the image of a geometrically irreducible variety is geometrically irreducible.}. Also, if one follows the intuition that $C_n$ should parametrize $K$-points obtained following the procedure of ``taking a $K$-point, then taking another $K$-point on its fiber, then switching fibrations'', we see how, starting from $P$, which is the only $K$-point on both $\pi_1^{-1}(\pi_1(P))$ and $\pi_2^{-1}(\pi_2(P))$, we are immediately stuck.
\end{remark}

Let us prove now that the $k$-scheme $\widetilde{f}_n^{-1}(E)$ does not contain any $k$-pseudo split subschemes. We will use the (contrapositive of the) following results.
\begin{lemma}\label{LemWPS}
Let $X$ be a $k$-scheme. If $X$ contains a pseudo-split $k$-subscheme, then, for each $\gamma \in \Gamma_k$, there exists an extension of fields $F/k$ such that the algebraic closure of $k$ in $F$ is (isomorphic, as a $k$-extension, to) $\overline{k}^{<\gamma>}$, and $X(F)\neq \emptyset$.
\end{lemma}
\begin{proof}
Suppose that $Y \subset X$ is a pseudo-split $k$-subscheme. The proof follows immediately upon taking $F$ to be the function field of a multiplicity-one irreducible component of $Y$ fixed by $\gamma$.
\end{proof}

We are going to give an explicit $\gamma \in \Gamma_k$ which makes the condition of the lemma above fail.

The $k$-scheme $\widetilde{f}_n^{-1}(E)$ is a union of subvarieties each of which possess a morphism to either $D_1$ or $D_2$. Let now $L_1$ and $L_2$ be the splitting fields of $D_1$ and $D_2$, and let $\gamma \in \Gamma_k$ be an element such that it projects to the non-trivial element of both $\Gal(L_1/K)$ and $\Gal(L_2/K)$ (recall that $L_1$ and $L_2$ are quadratic extensions of $K$, so there is a unique non-trivial element in both those groups). Note that such an element always exists: it suffices to take any lift of the non-trivial element of $\Gal(L_1/K)$ if $L_1=L_2$, or to take any lift of the element $(1,1)$ of $\Z/2\Z \times \Z/2\Z \cong \Gal(L_1/K) \times \Gal(L_2/K) = \Gal(L_1L_2/K)$ if $L_1 \neq L_2$. 
 
 We have now that, for any extension $F$ of $k$ such that the algebraic closure of $k$ in $F$ is (isomorphic to) $\overline{k}^{<\gamma>}$, $D_1(F)=D_2(F)=\emptyset$. In fact, since we have morphisms $D_1 \to \Spec L_1$ and $D_2 \to \Spec L_2$, we have that, for each $k$-extension $F'$ such that $D_1(F') \neq \emptyset$ (resp.\ $D_2(F') \neq \emptyset$), $L_1 \subseteq F'$ (resp.\ $L_2 \subseteq F'$). However, $F \nsubseteq L_1$ and $F \nsubseteq L_2$, because otherwise we would have that one among $L_1$ or $L_2$ would be contained in $\overline{k}^{<\gamma>}$, which is impossible by our choice of $\gamma$, which restricts to a non-trivial involution of both $L_1$ and $L_2$.
 It follows that for any such $F$, $\widetilde{f}_n^{-1}(E)(F)=\emptyset$. In particular, by Lemma \ref{LemWPS},  this proves that $\widetilde{f}_n^{-1}(E)$ does not contain any pseudo-split $k$-subschemes, as wished (with the condition of the lemma failing for the specific $\gamma$ we constructed).

Finally, let us prove that $\widetilde{f}_n^{-1}(\eta)$ does not contain any $k(\eta)$-pseudo-split subschemes. 

Recall that $\eta$ is the generic point of $E$, which is isomorphic to $\P^1$. In particular, we have that $k(\eta) \cong k(t)$, where $t$ is a transcendental parameter. We have the following short exact sequence (which arises from the identification of the Galois group of the extension $\overline{k}(t)/k(t)$ with that of $\overline{k}/k$):
\[
1 \to \Gal (\overline{k(t)}/\overline{k}(t)) \to \Gamma_{k(t)} = \Gamma_{k(\eta)} \to \Gamma_k \to 1.
\]
Let now $\gamma'\in \Gamma_{k(\eta)}$ be any lift of $\gamma\in \Gamma_{k}$. We claim that, for any $k(\eta)$-extension $F$ such that the algebraic closure of $k(\eta)$ in $F$ is (isomorphic to) $\overline{k(\eta)}^{<\gamma'>}$, we have that
$(\widetilde{f}_n^{-1}(\eta))(F)= \emptyset$. Note that, by Lemma \ref{LemWPS}, this would conclude the proof of the fact that $\widetilde{f}_n^{-1}(\eta)$ does not contain any $k(\eta)$-pseudo-split subschemes, and, hence, the proof of Lemma \ref{MainLemmaCounterexample}.

To prove the claim, assume that there is such an $F$ with $(\widetilde{f}_n^{-1}(\eta))(F) \neq \emptyset$. Note that there is a natural morphism of $k$-schemes $\widetilde{f}_n^{-1}(\eta) \to \widetilde{f}_n^{-1}(E)$, so we deduce that there exist an $F$-point in $\widetilde{f}_n^{-1}(E)$. Noting that the algebraic closure of $k$ in $F$ is isomorphic to the algebraic closure of $k$ in $\overline{k(\eta)}^{<\gamma'>}$, which is equal to $\overline{k}^{\gamma'}= \overline{k}^{\gamma}$, we reach our contradiction, since we already proved that, for such an $F$, $\widetilde{f}_n^{-1}(E)$ has no $F$-points.
\end{proof}

\subsection*{Acknowledgements}
We thank Daniel Loughran, Alexei Skorobogatov and Rosa Winter for useful comments and suggestions on early drafts, and we thank Gregory Sankaran and David Bourqui for feedback on the version appearing in the second author's PhD thesis at the University of Bath. We are grateful to the anonymous referees for their reports.
The first author was supported initially by Scuola Normale Superiore and Université Paris-Saclay (as a PhD student), and then by the Max Planck Institute at Bonn.
The second author was supported initially by an EPSRC studentship at the University of Bath and later by the Heilbronn Institute for Mathematical Research.

\end{document}